\documentclass[10pt,a4paper]{article}
\usepackage{amsmath,amsfonts,amsthm,amsopn,color,amssymb,enumitem}
\usepackage{palatino}
\usepackage{graphicx}
\usepackage[colorlinks=true]{hyperref}
\hypersetup{urlcolor=blue, citecolor=red, linkcolor=blue}
\usepackage{mathrsfs}

\newcommand{\eps}{\varepsilon}

\newcommand{\R}{\mathbb{R}}

\newcommand{\RT}{{\mathbb{R}^3}}

\renewcommand{\le}{\leqslant}
\renewcommand{\ge}{\geqslant}
\renewcommand{\a }{\alpha }

\newcommand{\g }{\gamma }

\renewcommand{\l }{\lambda}
\newcommand{\n }{\nabla }
\newcommand{\s }{\sigma }

\renewcommand{\O}{\Omega}

\newcommand{\G}{\Gamma}

\renewcommand{\H}{H^1_0(\O)}

\newcommand{\N}{\mathbb{N}}

\def\bbm[#1]{\mbox{\boldmath $#1$}}
\newcommand{\beq }{\begin{equation}}
\newcommand{\eeq }{\end{equation}}

\newtheorem{theorem}{Theorem}[section]
\newtheorem{lemma}[theorem]{Lemma}

\newtheorem{proposition}[theorem]{Proposition}

\title{
Generalized Schr\"odinger-Poisson type systems \footnote{The first two authors
are supported by M.I.U.R. - P.R.I.N. ``Metodi variazionali e
topologici nello studio di fenomeni non lineari''}}

\author{Antonio Azzollini \thanks{Dipartimento di Matematica ed Informatica, Universit\`a degli
Studi della Basilicata,  Via dell'Ateneo Lucano 10, I-85100 Potenza,
Italy, e-mail: {\tt antonio.azzollini@unibas.it}}
 \; \& \;
 Pietro d'Avenia\thanks{Dipartimento di Matematica, Politecnico di
Bari, Via E. Orabona 4, I-70125 Bari, Italy, e-mail: {\tt
p.davenia@poliba.it}}
\; \& \;
Valeria Luisi}

\date{}

\begin{document}

\maketitle

\begin{abstract}
In this paper we study the boundary value problem
\[
\left\{
\begin{array}{ll}
-\Delta u+ \eps q\Phi f(u)=\eta|u|^{p-1}u & \text{in } \Omega,  \\
- \Delta \Phi=2 qF(u)& \text{in } \Omega,   \\
u=\Phi=0 & \text{on }\partial \Omega,
\end{array}
\right.\] where $\Omega \subset \mathbb{R}^3$ is a smooth bounded
domain, $1 < p < 5$, $\eps ,\eta= \pm 1$, $q>0$, $f:\R\to\R$ is a
continuous function and $F$ is the primitive of $f$ such that
$F(0)=0.$ We provide existence and multiplicity results assuming on
$f$ a subcritical growth condition. The critical case is also
considered and existence and nonexistence results are proved.
\end{abstract}

\noindent{\small \textit{Keywords:} Schr\"odinger-Poisson
equations, variational methods, mountain pass.\newline\textit{2000
MSC:} 35J20, 35J57, 35J60.}

\section{Introduction}

This paper deals with the following problem

\begin{equation} \label{eq:P}\tag{$\mathcal P$}
\left\{
\begin{array}{ll}
-\Delta u+ \eps q\Phi f(u)=\eta|u|^{p-1}u &\text{in } \Omega, \\
-\Delta \Phi=2 qF(u) & \text{in }\Omega, \\
u=\Phi=0 & \text{on }\partial \Omega,
\end{array}
\right.
\end{equation}
where $\Omega \subset \mathbb{R}^3$ is a bounded domain with smooth
boundary $\partial\O$, $1 < p < 5$, $q >0$, $\eps,\eta = \pm 1$,
$f:\R\to\R$ is a continuous function
and $F(s)=\int_0^s f(t)\, dt.$\\
When the function $f(t)=t$ and $\eps=\eta =1$, this system
represents the well known Schr\"odinger-Poisson (or
Schr\"odinger-Maxwell) equations, briefly SPE, that have been
widely studied in the recent past. In the pioneer paper of Benci and
Fortunato \cite{BF}, the \emph{linear} version of SPE (where $\eta=0$)
has been approached as an
eigenvalue problem. In \cite{PS} the authors have proved the
existence of infinitely many solutions for SPE when $p>4$, whereas
in \cite{RS} an analogous result has been found for almost any $q>0$
and $p\in\, ]2,5[$. A multiplicity result has been obtained in
\cite{S} for any $q>0$ and $p$ sufficiently close to the critical
exponent 5, by using the abstract Lusternik-Schnirelmann theory.
For the sake of completeness we mention also \cite{PS0} where
Neumann condition on $\Phi$ is assumed on $\partial \Omega$,
\cite{A} and the references within for results on SPE in unbounded
domains and \cite{DPS1,DPS2} for the Klein-Gordon-Maxwell system in
a bounded domain.
\\
If $f(t)=t$ and $\eps =-1$,  the system is equivalent to a nonlocal
nonlinear problem related with the following well known Choquard
equation in the whole space $\R^3$
    \begin{equation*}
        \Delta u + u -\left(\frac 1 {|x|}*u^2\right)u=0.
    \end{equation*}
We refer to \cite{Lieb,Lions} for more details on the Choquard
equation and to \cite{Mu} for a recent result on a system in $\R^3$
strictly related with ours.\\
Up to our knowledge, problem \eqref{eq:P} has not been investigated
when a more general function $f$ appears instead of the identity.
Since problem \eqref{eq:P} possesses a
variational structure, our aim is to find weak assumptions on $f$ in
order to apply the usual variational techniques. In particular, the
first step in a classical approach to such a type of systems
consists in the use of the reduction method. To this end, we need to
assume suitable growth conditions on $f$ which allow us to invert
the Laplace operator and thus to solve the second equation of the
system. Then, after we have reduced the problem to \emph{a single
equation}, we find critical points of a one variable functional,
checking geometrical and compactness assumptions of the Mountain
Pass Theorem. If on one hand a suitable use of some {\it a priori}
estimates makes quite immediate to show that geometrical hypotheses
are verified (at least for small $q$), on the other some technical
difficulties arise in getting boundedness for the Palais-Smale
sequences. If $\eta=1$, we use a suitable truncation
argument based on an idea of Berti and Bolle \cite{BB} and Jeanjean
and Le Coz \cite{JC} (see also \cite{ADP,K2}) and we are
able to show the existence of a bounded Palais-Smale sequence of the
functional taking $q$ sufficiently small. If we had the sufficient
compactness, we would conclude by extracting any strongly convergent
subsequence from this bounded Palais-Smale sequence. However the
growth hypothesis we assume on $f$ does not permit to deduce
compactness on the variable $\Phi$. Indeed the exponent $4$ turns
out to be \emph{critical} and in this sense we are justified to refer to
\eqref{ipotesi} as the {\it critical growth
condition} for the function $f.$ \\
The first result in this paper is the following.

\begin{theorem} \label{main}
Let $f:\mathbb{R} \rightarrow \mathbb{R}$ be a
continuous function satisfying
\begin{equation} \label{ipotesi} \tag{$\mathbf{f}$}
|f(s)| \le c_1+c_2 |s|^{4}
\end{equation}
for all $s\in\R.$ Then, there exists $\bar{q}>0$ such that for all
$0 < q \le \bar{q}$ problem
\begin{equation}
\tag{$\mathcal P_f$}
\left\{
\begin{array}{ll}
-\Delta u+ \eps q\Phi f(u)=|u|^{p-1}u &\text{in } \Omega, \\
-\Delta \Phi=2 qF(u) & \text{in }\Omega, \\
u=\Phi=0 & \text{on }\partial \Omega,
\end{array}
\right.
\label{P}
\end{equation}
has at least a nontrivial
solution.
\end{theorem}

Inspired by \cite{Mu}, in the second part of this paper we  study
\eqref{eq:P} with $\eps=\eta=-1$ and $f(s)= |s|^{r-2}s.$ The system
becomes
\begin{equation} \label{Pq}\tag{$\mathcal P_r$}
\left\{
\begin{array}{ll}
    -\Delta u - q |u|^{r-2} u \Phi + |u|^{p-1} u=0
    &
    \hbox{in } \Omega,\\
    -r\Delta \Phi=2q |u|^r
    &
    \hbox{in } \Omega,\\
    u=\Phi=0
    &
    \hbox{on } \partial \Omega.
\end{array}
\right.
\end{equation}
where $1< r $ and $1<p<5$.

In this situation the contrasting nonlocal and local nonlinear terms
perturb the functional in a way which in some sense recalls the
typical concave-convex power-like nonlinearity. As a consequence,
the functional's geometry depends on the ratio of magnitude between
$r$ and $p$. We get the following result.
    \begin{theorem}\label{th:r}
        If $\frac {p+1}2<r<5$, problem \eqref{Pq}
        has infinitely many solutions for any $q>0.$\\
        If $r=\frac{p+1}2,$  then there exists an increasing sequence $(q_n)_n$
        such that, if $q\ge q_n$, problem \eqref{Pq} has at
        least $n$ couples of solutions.\\
        If $1<r<\frac{p+1}2,$ then there exists an increasing sequence $(q_n)_n$
        such that, if $q\ge q_n$, problem \eqref{Pq} has at
        least $2n$ couples of solutions.\\
        Finally, if $p<5\le r$, then the problem has no nontrivial solution.
    \end{theorem}
The paper is organized as follows: in Section \ref{VT} we introduce
the functional setting where we study the problem \eqref{P} and the
variational tools we use; in Section \ref{proof} we provide the
proof of Theorem \ref{main}; in Section \ref{sec:positive} we
consider the system \eqref{Pq} and prove Theorem \ref{th:r}.

Throughout the paper we will use the symbols $H^{-1}$ to denote the
dual space of $H^1_0(\Omega)$, $\langle \cdot,\cdot \rangle$ to
denote the duality between $H^1_0(\Omega)$ and $H^{-1}$ and
$\|u\|_{H^1_0}:=\left(\int_{\O}|\nabla u|^2\right)^{\frac 1 2}$ for the norm on $H^1_0(\Omega)$.
Moreover $\|\cdot\|_p$ will denote the usual $L^p(\O)$-norm.

We point out the fact that in the sequel we will use the symbols
$C,$ $C_1,$ $C_2,$ $C_3$ and so on, to denote positive constants
whose value might change from line to line.

\section{Variational tools}\label{VT}

Standard arguments can be used to prove that problem \eqref{P} is
variational and the related $C^1$ functional $J_q:H^1_0(\Omega) \times
H^1_0(\Omega) \rightarrow \mathbb{R}$ is given by
\begin{equation*}
J_q(u,\Phi)=\frac{1}{2}\int_{\Omega} |\nabla
u|^2dx-\frac{\eps}{4}\int_{\Omega} |\nabla \Phi|^2dx+\eps q \int_{\Omega}
F(u)\Phi dx-\frac{1}{p+1}\int_{\Omega} |u|^{p+1}dx.
\end{equation*}

Since, by \eqref{ipotesi}, $F:L^{6}(\Omega) \to L^{\frac{6}{5}}(\Omega)\hookrightarrow H^{-1}$, then certainly we have that for all $u \in
H^1_0(\Omega)$ there exists a unique $\Phi_u \in
H^1_0(\Omega)$ which solves
\begin{equation*}
-\Delta \Phi=2 qF(u) \label{2eq}
\end{equation*}
in $H^{-1}.$
In particular, we are allowed to consider the following map
\begin{equation*}
u \in L^{6}(\Omega) \mapsto \Phi_u \in H^1_0(\Omega)
\end{equation*}
which is continuously differentiable by the Implicit Function Theorem applied to $\partial_\Phi H$ where, for any $(u,\Phi)\in L^{6}(\Omega)\times\H$,
    \begin{equation*}
        H(u,\Phi) :=
        \frac 1 4 \int_{\O}|\n \Phi|^2 dx -q\int_{\O}F(u)\Phi dx.
    \end{equation*}

Since, for every $u\in\H$,
\begin{equation}
\partial_{\Phi} J_q(u,\Phi_{u})=- \eps\partial_\Phi H(u,\Phi_{u})
=0,
\label{eq:J'}
\end{equation}
then
\begin{equation}
        \int_\O |\n \Phi_u|^2\, dx=2 q\int_\O F(u)\Phi_u\, dx \label{svista}
    \end{equation}
and
\begin{equation}
\int_{\Omega} F(u)\Phi_u dx \ge 0.
\label{eq:pos}
\end{equation}
Moreover, we have the following estimates.
\begin{lemma}
For every $u\in\H$
\begin{equation}
\left(\int_{\Omega} |\n\Phi_u|^2dx\right)^{1/2}\le Cq
\left(\int_{\Omega} |F(u)|^{6/5}dx\right)^{5/6} \label{Holder2}
\end{equation}
and
\begin{equation}\label{eq:control}
        \int_\O F(u)\Phi_u dx
        \le q( C_1\|u\|_{ 6/5}^2+C_2\|u\|^{10}_{6}).
\end{equation}
\end{lemma}
\begin{proof}
By Holder inequality we have
\begin{equation} \int_{\Omega} F(u)\Phi_u dx\le \left(\int_{\Omega}
|F(u)|^{6/5}dx\right)^{5/6}\left(\int_{\Omega}
|\Phi_u|^{6}dx\right)^{1/6} .\label{Holder1}
\end{equation}
Then, from \eqref{svista}, by using Sobolev embedding $\H\hookrightarrow L^6(\O)$, we get
\begin{align*}
\int_{\Omega} |\nabla \Phi_u |^2 dx&= 2 q \int_{\Omega} F(u) \Phi_u dx
\le 2 q \left(\int_{\Omega}
|F(u)|^{6/5}dx\right)^{5/6}\left(\int_{\Omega}
|\Phi_u|^{6}dx\right)^{1/6}
\\ &
\le C q  \left(\int_{\Omega} |F(u)|^{6/5}dx\right)^{5/6}\left(\int_{\Omega}
|\nabla \Phi_u|^{2}dx \right)^{1/2}
\end{align*}
and then we get \eqref{Holder2}.\\
By \eqref{Holder1} and \eqref{Holder2} we have
\begin{equation*}
        \int_{\Omega} F(u) \Phi_u dx \le C q \left(\int_{\Omega}|F(u)|^{6/5}dx\right)^{5/3}
    \end{equation*}
and then, since for \eqref{ipotesi} it is
    \begin{equation}
        |F(s)|^{6/5} \le C_1 |s|^{6/5}+C_2|s|^{6}, \label{crescita}
    \end{equation}
for all $s \in \mathbb{R}$, we get \eqref{eq:control}.
\end{proof}

Equation \eqref{svista} allows us to define on $H^1_0(\Omega)$ the $C^1$ one variable functional
        \begin{equation*}
        I_q(u):=J_q(u,\Phi_u)=
        \frac{1}{2}\int_{\Omega} |\nabla u|^2dx+ \eps \frac{q}{2}
        \int_{\Omega} F(u)\Phi_u dx-\frac{1}{p+1}\int_{\Omega} |u|^{p+1}dx.
        \end{equation*}

By using standard variational arguments as those in \cite{BF}, the
following result can be easily proved.
\begin{proposition}
Let $(u,\Phi)\in H^1_0(\Omega)\times
H^1_0(\Omega)$, then the following propositions are equivalent:
\begin{enumerate}[label=(\alph*), ref=\alph*]
\item $(u,\Phi)$ is a critical point of functional $J_q$;
\item $u$ is a critical point of functional $I_q$ and
$\Phi=\Phi_u$.
\end{enumerate}
\end{proposition}

So we are led to look for critical points of $I_q.$ To this end,
we need to investigate the compactness property of its
Palais-Smale
sequences.\\
It is easy to see that the standard arguments used to prove boundedness do
not work. Indeed, assuming that $(u_n)_n\in (\H)^{\N}$ is a
Palais-Smale sequence, namely $(I_q(u_n))_n$ is bounded and
$I_q'(u_n)\to 0$ in $H^{-1}$, we obtain the following
inequality
    \begin{equation}\label{eq:p-s}
        I_q(u_n) - \frac{1}{p+1}\langle I_q'(u_n),u_n\rangle
        \le C_1 + C_2 \|u_n\|.
    \end{equation}
Since, by \eqref{eq:J'},
    \begin{equation*}
        \langle I'_q(u_n),u_n\rangle = \langle \partial_u
        J_q(u_n,\Phi_{u_n}),u_n\rangle,
    \end{equation*}
from \eqref{eq:p-s} we get
    \[
    \left(\frac 1 2 -\frac 1 {p+1}\right)\int_\O |\n u_n|^2 dx+
    \eps \frac q 2\int_\O F(u_n)\Phi_{u_n} dx
    -\eps\frac q {p+1} \int_\O f(u_n)u_n
    \Phi_{u_n } dx \le C_1+ C_2\|u_n\|.
    \]
In the classical SPE ($\eps=1$ and $f(t)=t$) we should deduce the
boundedness of the sequence $(u_n)_n$ for $p\ge 3$. In our general
situation we need a different approach.

Let $T>0$ and $\chi :[0,+\infty[ \rightarrow [0,1]$ be a smooth
function such that $\|\chi^{\prime}\|_{L^{\infty}} \le 2$ and
\[ \chi(s)=\left\{
\begin{array}{ll}
1 &\text{if }0 \le s \le 1 \\
0 &\text{if }s \ge 2.
\end{array} \label{chi}
\right.
\]
We define a new functional $I^T_q:H^1_0(\Omega) \rightarrow
\mathbb{R}$ as follows
\begin{equation}
I^T_q(u)=\frac{1}{2}\int_{\Omega} |\nabla u|^2dx+ \eps
\frac{q}{2}\chi \left(\frac{\|u\|_{H^1_0}}{T}\right) \int_{\Omega}
F(u)\Phi_u dx-\frac{1}{p+1}\int_{\Omega} |u|^{p+1}dx
\label{functional2}
\end{equation}
for all $u \in H^1_0(\Omega)$.
We are going to find a critical point $u \in H^1_0(\Omega)$ of this new functional such that $\|u\|_{H^1_0} \le T$ in order to get solutions of our problem.

\section{Proof of Theorem \ref{main}}\label{proof}

We prove that the functional $I^T_q$ satisfies Mountain Pass
geometrical assumptions. More precisely, we have the following
result.

\begin{lemma}\label{lemma1}
Under hypothesis \eqref{ipotesi}, there exists $\bar{q}\in
\R_+\cup\{+\infty\}$ such that for all $0<q < \bar{q}$ functional
$I^T_q$ satisfies:
\begin{enumerate}[label=(\roman*), ref=(\roman*)]
\item \label{31i}$I^T_q(0)=0$;
\item \label{31ii}there exist constants
$\rho,\alpha>0$ such that
\begin{equation*}
I^T_q(u)\ge \alpha \qquad \text{for all } u \in H^1_0(\Omega) \quad
\text{with } \|u\|_{H^1_0}=\rho ;
\end{equation*}
\item \label{31iii} there exists a function $\bar{u} \in H^1_0(\Omega)$
with $\|\bar{u}\|_{H^1_0}> \rho$ such that $I^T_q(\bar{u}) < 0$.
\end{enumerate}
\end{lemma}

\begin{proof}
Property \ref{31i} is trivial.\\
To prove \ref{31ii} we distinguish two cases.\\
If $\eps =1$, by \eqref{eq:pos}, for every $q>0$ we deduce that
\[
I^T_q(u)\ge \frac{1}{2} \rho^2 - C \rho^{p+1} \ge \a >0
\]
for suitable $\rho, \a >0$.\\
 If $\eps =-1$, by using \eqref{eq:control} and the immersion of $H^1_0(\Omega)$
into $L^p(\Omega)$ spaces, we get
\begin{align*}
I^T_q(u) & \ge \frac{1}{2} \|u\|^2_{H^1_0}-\frac{q}{2} \int_{\Omega}
F(u)\Phi_u dx-\frac{1}{p+1}\int_{\Omega}
|u|^{p+1}dx  \\
&  \ge \frac{1}{2} \|u\|^2_{H^1_0}-\frac{q^2}{2}(C_1\|u\|_{
6/5}^2+C_2\|u\|^{10}_{6}) -\frac{1}{p+1}
\|u\|^{p+1}_{p+1}  \\
& \ge \frac{1}{2} \|u\|^2_{H^1_0}-\frac{q^2}{2}(C_1
\|u\|^{2}_{H^1_0}+C_2 \|u\|^{10}_{H^1_0}) -\frac{C}{p+1}
\|u\|^{p+1}_{H^1_0}.
\end{align*}
Thus, if $q$ is such that $q^2C_1<1$ and $\rho$ is small enough,
there exists $\alpha
>0$ such that $I^T_q(u) \ge \alpha$ for all $u \in H^1_0(\Omega)$
with $\|u\|_{H^1_0}=\rho$.
\\
To prove \ref{31iii}, let us consider $u \in H^1_0(\Omega)$, $u \neq 0$ and $t >
\frac{2T}{\|u\|_{H^1_0}}$ such that $\chi \left(\frac{ \|tu\|_{H^1_0}}{T}\right)=\chi \left(\frac{t \|u\|_{H^1_0}}{T}\right)=0$. Then, we have
\[
I^T_q(tu) =\frac{t^2}{2} \int_{\Omega} |\nabla u|^2 dx -
\frac{t^{p+1}}{p+1}\int_{\Omega} |u|^{p+1}dx
\]
and so, for
$t$ large enough, $I^T_q(tu)$ is negative. 
\end{proof}

Thus we can complete the proof of Theorem \ref{main}.

\begin{proof}[Proof of Theorem \ref{main}]
Lemma \ref{lemma1} allows us to define, for $q<\bar q$,
\begin{equation}
m^T_q=\inf_{\gamma \in \Gamma}\sup_{t \in [0,1]}I^T_q(\gamma(t))>0
\label{infsup2}
\end{equation}
where
\begin{equation*}
\Gamma=\{ \gamma \in \mathcal{C}([0,1],H^1_0(\Omega))\mid
\gamma(0)=0, I^T_q(\gamma(1))<0 \} .
\end{equation*}
Certainly there exists a Palais-Smale sequence at mountain pass
level $m^T_q$, that is a sequence $(u_n)_n$ in $H^1_0(\Omega)$
such that
\begin{equation}
I^T_q(u_n) \rightarrow m^T_q \label{PS1}
\end{equation}
and
\begin{equation}
(I^T_q)^{\prime}(u_n) \rightarrow 0 .\label{PS2}
\end{equation}
As a first step we prove that there exists $\bar T>0$ and $\tilde q >0$ such that for any
$0<q\le\tilde q$ there exists a Palais-Smale sequence $(u_n)_n$ of
$I_q$ at the level $m^{\bar T}_q$ such that, up to a subsequence,
$\|u_n\|_{\H}\le \bar T$ for any $n\in\N$.\\
Let $T>0$, $q>0$ and $(u_n)_n$ in $H^1_0(\Omega)$ be a
Palais-Smale sequence of $I^T_q$ at level $m^{T}_q$.\\
We make some preliminary computations.
The first is an estimate on the mountain pass level $m_q^{T}$.\\
Let $u \in
H^1_0(\Omega), u \neq 0$ and $\bar{t}>0$ be such that the path
$\bar\g(t)=t\bar t u$ belongs to $\G.$ For all $t \in [0,1]$ it is
\[
I^{T}_q(\bar\g(t))  = \frac{t^2}{2}\int_{\Omega} |\n\bar{t}
u|^2
dx
+\eps \frac{q}{2}\chi \left(\frac{t \|\bar{t}u\|_{H^1_0}}{\bar
T}\right) \int_{\Omega} F(t\bar{t}u)\Phi_{t\bar{t}u}
dx-\frac{t^{p+1}}{p+1}\int_{\Omega} |\bar{t}u|^{p+1}dx.
\]
From \eqref{eq:control} we obtain
\begin{align*}
\max_{t\in[0,1]}I^{T}_q(\bar\g(t)) &\le \max_{t\in
[0,1]}\left(C_1 t^2\int_\O|\n u|^2\, dx -
C_2 t^{p+1} \int_\O|u|^{p+1}\, dx\right)\\
&\qquad +q^2\max_{t\in [0,1]}\left[\chi\left(\frac{t\bar t
\|u\|_{H^1_0}}{{T}}\right)\left(C_3 t^2\|\bar t
u\|^2_{H^1_0}+C_4
t^{10}\|\bar t u\|^{10}_{H^1_0}\right)\right]\\
&\le C + q^2 (C_5 {T}^2 + C_6 {T}^{10}).
\end{align*}
Then, from \eqref{infsup2} we get
\begin{equation}
m^{T}_q \le C + q^2 (C_5 {T}^2 + C_6 {T}^{10})
.\label{stimalivello}
\end{equation}
From \eqref{functional2} we deduce that
\begin{equation*}
\langle (I^{\bar T}_q)^{\prime}(u_n), u_n
\rangle=\int_{\Omega}|\nabla u_n|^2 dx+{\mathcal A}_n+{\mathcal
B}_n+{\mathcal C}_n-\int_{\Omega}|u_n|^{p+1}dx
\end{equation*}
where
\begin{equation*}
{\mathcal A}_n=\eps \frac q 2
\chi^{\prime}\left(\frac{\|u_n\|_{H^1_0}}{T}\right)\frac{\|u_n\|_{H^1_0}}{T}\int_{\Omega}F(u_n)\Phi_{u_n}dx
\end{equation*}
\begin{equation*}
{\mathcal B}_n=\eps \frac{q}{2}\chi
\left(\frac{\|u_n\|_{H^1_0}}{{\bar
T}}\right)\int_{\Omega}f(u_n)u_n\Phi_{u_n} dx
\end{equation*}
and
\begin{equation*}
{\mathcal C}_n=\eps \frac{q}{2}\chi
\left(\frac{\|u_n\|_{H^1_0}}{{\bar T}}\right)\int_{\Omega}F(u_n)
\Phi_{u_n}^{\prime}[u_n] dx .
\end{equation*}
Since the functional
    \begin{equation*}
        S(u)=\int_\O|\n \Phi_u|^2 dx - 2 q\int_\O F(u)\Phi_u dx, \; u\in\H
    \end{equation*}
is identically equal to 0, we have that
    \begin{equation*}
        0=\frac{1}{2}\langle S'(u),u\rangle = \int_{\O} (\n \Phi_u|\n \Phi'_u[u])dx - q \int_\O f(u)\Phi_u u dx- q \int_\O F(u) \Phi'_u[u]dx.
    \end{equation*}
On the other hand, multiplying the second equation of \eqref{P} by
$\Phi'_u[u]$ and integrating, we have that
    \begin{equation*}
        \int_\O (\n \Phi_u|\n \Phi'_u[u])dx=2q \int_\O F(u) \Phi'_u[u]dx
    \end{equation*}
and then
$$\int_\O F(u) \Phi'_u[u]dx=\int_\O f(u)\Phi_u u dx.$$
We deduce that ${\mathcal B}_n={\mathcal C}_n$ for any $n\in \mathbb{N}.$
By using \eqref{PS1} and \eqref{PS2} we get
\begin{align}
m^{T}_q+o_n(1)\|u_n\|_{H^1_0}+o_n(1)&=I^{T}_q(u_n)-\frac{1}{p+1}\langle (I^{T}_q)^{\prime}(u_n), u_n
\rangle\nonumber\\
&=\frac{p-1}{2(p+1)}\|u_n\|^2_{H^1_0}+{\mathcal D}_n
-\frac{1}{p+1}{\mathcal A}_n-\frac{2}{p+1}{\mathcal B}_n,
\label{eq:ps}
\end{align}
where
    $${\mathcal D}_n=\eps \frac{q}{2}\chi \left(\frac{ \|u_n\|_{H^1_0}}{{T}}\right)
        \int_{\Omega} F(u_n)\Phi_{u_n} dx.$$
For \eqref{ipotesi} and \eqref{eq:control} we have also the
following estimate
\begin{equation*}
\max \left( |{\mathcal A}_n|,|{\mathcal B}_n|,|{\mathcal
D}_n|\right)\le q^2(C_1 {T}^2 + C_2 {T}^{10}).
\end{equation*}
We show that, if ${T}$ is sufficiently large, then
$\limsup_{n}\|u_n\|_{H^1_0}\le {T}$. \\ By contradiction, we
will assume that there exists a subsequence (relabeled $(u_n)_n$)
such that for all $n\in\mathbb{N}$ we have $\|u_n\|_{H^1_0}> {T}$. By our contradiction hypothesis, \eqref{stimalivello} and
\eqref{eq:ps}, we obtain, for $n$ large enough,
\begin{equation*}
{T}^2 -\s {T} \le \|u_n\|^2_{H^1_0} - \s
\|u_n\|_{H^1_0}\le C + q^2 (C_1 {T}^2+C_2 {T}^{10}),
\end{equation*}
with $\s >0$ small. If ${T}^2 - \s {T} >C$, we can find
$\tilde q$ such that for any $q\le
\tilde q$ the previous inequality turns out to be a contradiction.
The contradiction arises from the assumption that
$\limsup_{n}\|u_n\|_{H^1_0}\ge {T}$. So we have that
 the sequence $(u_n)_n$ possesses a subsequence which is bounded in the $\H$ norm by ${T}$ and such that,
 for every $n\in\mathbb{N}$,
 $I^{T}_q(u_n)$ coincides with $I_q(u_n)$.\\
The last step is to prove that there exists $\bar q >0$ such that
for any $0<q\le\bar q$ there exists a Palais-Smale sequence
$(u_n)_n$ of $I_q$ which is, up to a subsequence, weakly convergent
to  a nontrivial critical point of
$I_q$.\\
Let $\bar T$ and $\tilde q$ be given by the first step, and consider
any $0< q \le\tilde q$. We know that there exists a Palais-Smale
sequence of the functional $I_q$ at the level $m_q:= m_q^{\bar T}$,
such that
    \begin{equation}\label{eq:bou}
         \|u_n\|_{\H}\le \bar T, \hbox{ for any $n\in\N$.}
    \end{equation} Up to subsequences, there exist $u_0\in\H$
and $\Phi_0\in\H$ such that
\begin{align}
u_n &\rightharpoonup u_0\text{ in } H^1_0(\Omega)\label{eq:uzero}\\
\Phi_{u_n} &\rightharpoonup {\Phi_0} \text{ in }
H^1_0(\Omega)\label{fibarra} .
\end{align}
By \eqref{ipotesi} and \eqref{eq:uzero} we also have
\begin{align}
F(u_n) &\rightharpoonup F(u_0) \quad \text{in } L^{\frac{6}{5}}(\O),\label{convl}\\
f(u_n)u_n&\rightharpoonup f(u_0)u_0 \quad \text{in }
L^{\frac{6}{5}}(\O).\label{convv2}
\end{align}
Now we show that $\Phi_0=\Phi_{u_0}$ and that $(u_0,\Phi_0)$ is a weak nontrivial solution of \eqref{P}.\\
Let us consider a test function $\psi \in {C^{\infty}_{0}}(\Omega)$.
From the second equation of our problem we obtain
\begin{equation*}
\int_{\Omega}(\nabla \Phi_{u_n}| \nabla \psi ) dx= 2q \int_{\Omega}
F(u_n) \psi dx.
\end{equation*}
Passing to the limit and using \eqref{fibarra} and \eqref{convl}, we
have that
\begin{equation*}
\int_{\Omega}(\nabla {\Phi_0}| \nabla \psi ) dx= 2q \int_{\Omega}
F(u_0) \psi dx.
\end{equation*}
So ${\Phi_0}$ is a weak solution of $- \Delta \Phi =2q F(u_0)$,
and then, by uniqueness, it is ${\Phi_0}=\Phi_{u_0}$.\\
Since $(u_n)_n$ is a Palais-Smale sequence, for any $\psi \in
{C^{\infty}_{0}}(\Omega)$ we obtain that
    \begin{equation*}
        \int_{\O} (\n u_n|\n \psi)dx+\eps q \int_{\O} \Phi_{u_n}
        f(u_n)\psi dx=\int_{\O}
        |u_n|^{p-1}u_n\psi dx+ o_n(1).
    \end{equation*}
Passing to the limit, by \eqref{eq:uzero} and \eqref{convv2} we have
    \begin{equation*}
        \int_{\O} (\n u_0|\n \psi)dx+\eps q \int_{\O} \Phi_{u_0}
        f(u_0)\psi dx=\int_{\O}
        |u_0|^{p-1}u_0\psi dx
    \end{equation*}
that is $(u_0,\Phi_{u_0})$ is a weak solution of \eqref{P}.
\\
It remains to prove that $u_0\neq 0.$\\
Assume by contradiction that $u_0=0.$ By compactness we obtain that
$u_n\to 0$ in $L^{p+1}(\O).$ On the other hand, since $\langle
I_q'(u_n),u_n\rangle \to 0,$ we deduce that, up to subsequences,

    \begin{equation}\label{psseq}
        \lim_n \int_\O |\n u_n|^2\, dx=-\lim_n \eps q\int_\O f(u_n)u_n\Phi_{u_n}\, dx =:l_q\ge 0.
    \end{equation}
Of course $l_q>0$. Otherwise from \eqref{psseq} we would deduce that
$u_n\to 0$ in $\H$ and then $0<m_q=\lim_n I_q(u_n)=0$.\\
By
\eqref{eq:bou} we also have that $l_q\le\bar T.$ By using Holder
inequality, Sobolev inequalities, \eqref{ipotesi}, \eqref{Holder2} and \eqref{crescita}
we have that
    \begin{align*}
        -\eps q \int_\O f(u_n)u_n\Phi_{u_n}dx&\le Cq^2
\left(\int_{\Omega} |f(u_n)u_n|^{6/5}dx\right)^{5/6}
\left(\int_{\Omega}
|F(u_n)|^{6/5}dx\right)^{5/6}\\
&\le q^2\left( C_1\|u_n\|_{\H}^2+C_2\|u_n\|^{10}_{\H}\right).
    \end{align*}
Passing to the limit, by \eqref{psseq} we deduce that
    \begin{equation*}
        l_q\le q^2 \left( C_1l_q^2+C_2l_q^{10}\right)
    \end{equation*}
that is
    \begin{equation*}
       1\le q^2\left( C_1l_q+C_2l_q^{9}\right)\le q^2\left( C_1\bar T +C_2\bar T^{9}\right).
    \end{equation*}
We conclude observing that the previous inequality does not hold if
we take $q\le\bar q< \min\left\{\tilde q,(\frac 1 {C_1\bar T +C_2\bar T^{9}})^{\frac 1
2}\right\}$.

\end{proof}

\section{Proof of Theorem \ref{th:r}}\label{sec:positive}

This section deals with the study of the system \eqref{Pq}. We
divide the proof of Theorem \ref{th:r} in two parts, concerning
respectively the existence and the non existence of a solution
according to the value of $r$.

\subsection{The existence result}

Here, we suppose that $1<r<5$. In analogy to problem \eqref{P}, we
use a variational approach finding the solutions as critical points
of the $C^1$ functional
\[
I_{q,r} (u)=\frac{1}{2}\int_{\Omega} |\nabla u|^2dx
        -\frac{q}{2r}\int_{\Omega} |u|^r \Phi_u dx
        +\frac{1}{p+1}\int_{\Omega} |u|^{p+1}dx,
\]
where, for any $u\in\H$, $\Phi_u\in\H$ is the unique positive solution of
\begin{equation}
\left\{
\begin{array}{ll}
    -r\Delta \Phi=2q |u|^r
    &
    \hbox{in } \Omega,\\
    \Phi=0
    &
    \hbox{on } \partial \Omega.
\end{array}
\right.
\label{eq:seceq}
\end{equation}
We have the following estimates.
\begin{lemma}
For every $u \in H^1_0(\Omega)$
\begin{equation}\label{eq:estimate}
\| \Phi_u \|_{\H} \le  C \|u \|^r_{\frac{6}{5}r}
\end{equation}
and for any $k>0$ it is
    \begin{equation}\label{eq:upperest}
        \frac{kq}{r}\int_{\Omega} |u|^r \Phi_u dx\ge \frac{2q}{r}\int_\O
        |u|^{r+1}dx-\frac 1 {2k} \int_\O |\n u|^2 dx.
    \end{equation}
\end{lemma}
\begin{proof}
    The first part of the lemma is a consequence of the fact that,
    since $1< r< 5$, then for any $u\in\H$ the function
    $|u|^r\in L^{\frac 6 5}(\O)$ and we can argue as in Section
    \ref{VT} to define the map $\Phi$ and to deduce
    \eqref{eq:estimate}.\\
    In order to prove \eqref{eq:upperest}, we proceed as in
    \cite{Ru}: multiplying \eqref{eq:seceq} by $|u|$ and integrating
    we get
        \begin{align*}
                    \frac{2q}{r}\int_\O
        |u|^{r+1}dx&= \int_\O (\n \Phi_u|\n |u|)\,dx=\int_\O \left( \sqrt k\n\Phi_u\;\vline\;\frac1{\sqrt k}\n
        |u|\right)\,dx\\
        &\le \frac k 2 \int_\O|\n\Phi_u|^2 dx+\frac 1 {2k} \int_\O |\n |u||^2 dx.
        \end{align*}
    Inequality \eqref{eq:upperest} follows since it is
        \begin{equation*}
            \int_\O|\n\Phi_u|^2\,dx=\frac{2q}r\int_\O |u|^r\Phi_u\,
            dx.
        \end{equation*}
\end{proof}

In the following lemma we establish the compactness of the
Palais-Smale sequences of the functional $I_{q,r}$.

\begin{lemma}\label{le:ps}
If $1<r<5$ the functional $I_{q,r}$ satisfies the Palais-Smale condition.
\end{lemma}

\begin{proof}
Let $(u_n)_n$ be a a Palais-Smale sequence for the functional $I_{q,r}$,
namely $(I_{q,r}(u_n))_n$ is bounded and $I'_{q,r}(u_n)$ converges to zero in
$H^{-1}$.\\
We distinguish two cases.\\
If $r\ge (p+1)/2$, then
\begin{align*}
I_{q,r} (u_n)- \frac{I'_{q,r} (u_n)[u_n]}{p+1}  &=\frac{p-1}{2(p+1)}  \int_\O
|\nabla u_n|^2dx + q  \frac{2r-p-1}{2r(p+1)} \int_\O |u_n|^r
\Phi_{u_n} dx
\\
&\le C + o_n(1)\|u_n\|_{\H}
\end{align*}
and then $(u_n)_n$ is bounded.
\\
If $r < (p+1)/2$, then the boundedness of $(u_n)_n$ comes from the
boundedness of $(I_q(u_n))_n$ since the functional is coercive.
Indeed, if  we suppose that
$(u_n)_n$ diverges in the $H^1_0 (\O)-$norm, then, by \eqref{eq:estimate}
and the Lebesgue embedding $L^{p+1}(\Omega)\hookrightarrow L^{\frac{6}{5}r}(\Omega)$,
\[
I_{q,r}(u_n)\ge C_1 \| u_n \|_{\H}^2 - C_2 \| u_n
\|_{p+1}^{2r} + C_3 \| u_n \|_{p+1}^{p+1} \rightarrow +\infty.
\]
%
%
%
%
\end{proof}

Thus we can complete the proof of the Theorem \ref{th:r}.
\begin{proof}[Proof of Theorem \ref{th:r}]
We first suppose that $r$ is \emph{subcritical} and we deal with each case
separately.
\begin{center}
    \emph{case 1: $\displaystyle\frac{p+1}{2} < r< 5$}
\end{center}
We show that for any $q>0$ and for any finite dimensional subspace
of $\H$, the functional $I_{q,r}$ satisfies the assumptions of
\cite[Theorem 2.23]{AR}.\\
By \eqref{eq:estimate} and the Sobolev embedding $\H\hookrightarrow
L^{\frac 6 5 r}(\O)$ it is
\begin{equation*}
I_{q,r}(u)\ge \frac{1}{2} \int_\O |\nabla u|^2dx - C
\|u\|^{2r}_{\frac{6}{5}r}
\ge \frac{1}{2} \|u\|_{\H}^2 -
\bar C \|u\|_{\H}^{2r},  
\end{equation*}
then assumption $(I_1)$ holds since $I_{q,r}(u) \ge\a>0$ if $\|u\|_{\H}$
is sufficiently small. By Lemma \ref{le:ps}, also assumption $(I_3)$
holds. $(I_4)$ can be checked by an easy computation. In order to
prove $(I_7),$ we show that for any finite dimensional subspace $E$
of $\H$ there exists a ball $B_{\bar \rho}$ such that $I_{q,r}|_{E\cap
\partial B_{\bar\rho}}<0$.\\
Let $E$ be a finite dimensional subspace of $\H.$ It is easy to see
that for $\rho>0$ and $u\in\H$
\[
\Phi_{\rho u}=\rho ^r \Phi_u.
\]
Since in $E$ all the norms are equivalent, if $u\in E\cap
\partial B_1,$ by \eqref{eq:upperest} we have
\begin{align}
    I_{q,r}(\rho u)&=\frac{\rho^2}{2} \int_\O |\nabla u|^2dx - \frac{q}{2r}
    \rho^{2r} \int_\O |u|^r \Phi_{u} dx +
    \frac{\rho^{p+1}}{p+1}\int_{\Omega} |u|^{p+1}dx\nonumber\\
    &\le\frac{\rho^2}{2}-  \frac{q}{kr}\rho^{2r}\int_\O
        |u|^{r+1}\,dx+\frac 1 {4k^2}\rho^{2r} + \frac{\rho^{p+1}}{p+1}\int_{\Omega}
        |u|^{p+1}dx\nonumber\\
        &\le \frac{\rho^2}{2}-  (c_1(k) -c_2(k))\rho^{2r} +
        c_3\rho^{p+1}.\label{eq:negative}
\end{align}
Since $c_1(k)= O(1/k)$ and $c_2(k)= O(1/{k^2})$ for $k\to + \infty,$
we have that for a $k$ sufficiently large $c_1(k)- c_2(k)>0.$ So we
can take $\bar\rho>0$ as large as needed to have that
\eqref{eq:negative} is negative.
\begin{center}
    \emph{case 2: $\displaystyle\frac{p+1}{2}=  r$}
\end{center}
The proof is the same as in the previous case, except for $(I_7).$
In particular we can only prove that for any $E\subset\H$ finite
dimensional subspace there exists $\bar q>0$ such that for every $
q>\bar q$ and  $\bar\rho$ large enough, $I_{q,r}|_{E\cap
\partial B_{\bar\rho}}<0$. Indeed, as in
\eqref{eq:negative}, we have that for any $u\in E\cap
\partial B_1,$ it is
    \begin{align*}
        I_{q,r}(\rho u)&\le\frac{\rho^2}{2}-  \frac{q}{rk}\rho^{2r}
        \|u\|_{r+1}^{r+1}+\frac {\rho^{2r}} {4k^2} + \frac{\rho^{2r}}{2r}
        \|u\|_{2r}^{2r}\\
        &\le \frac{\rho^2}{2}-  (q c_1(k) -c_2(k)- c_3)\rho^{2r}
    \end{align*}
so that if $q$ and $\rho$ are sufficiently large, $I_{q,r}(\rho u)<0$.
\begin{center}
    \emph{case 3: $\displaystyle 1 < r < \frac{p+1}{2}$}
\end{center}
In this case it is easy to check that for any $E\subset\H$ finite
dimensional subspace there exists $q>0$ such that $I_{q,r}|_{E\cap
\partial B_1}<0$. In fact, for any $u\in E\cap
\partial B_1,$ we have that
    \begin{align*}
        I_{q,r}(\rho u)&\le\frac{\rho^2}{2}-  \frac{q}{kr}\rho^{2r}\int_\O
        |u|^{r+1}\,dx+\frac 1 {4k^2}\rho^{2r} + \frac{\rho^{p+1}}{p+1}\int_{\Omega}
        |u|^{p+1}dx\\
        &\le \frac{\rho^2}{2}-  (q c_1(k) -c_2(k)- c_3)\rho^{p+1}.
    \end{align*}
To complete the proof, by \cite[Corollary 2.24]{AR} we have just to
show that for any $q$ the functional $I_{q,r}$ is bounded from below.
Indeed, by \eqref{eq:estimate},
\begin{align*}
I_{q,r}(u)\ge &
\frac{1}{2} \int_\O |\nabla u|^2dx
 - C \| u \|_{\frac{6}{5}r}^{2r}
 + \frac{1}{p+1} \| u \|_{p+1}^{p+1}
\\
\ge&
\frac{1}{2} \int_\O |\nabla u|^2dx - C \| u \|_{p+1}^{2r} + \frac{1}{p+1} \| u \|_{p+1}^{p+1}\\
=&\frac{1}{2} \int_\O |\nabla u|^2dx + \| u \|_{p+1}^{2r}
\left(\frac{1}{p+1} \| u \|_{p+1}^{p+1-2r} - C \right)
\end{align*}
where the last quantity cannot diverge negatively.
\\
\end{proof}

It is quite natural to wonder if some nonexistence result can be
proved in the case $1 < r \le \frac{p+1}{2}$ and $q$ small.
Actually, it can easily be observed that the existence of at least a
solution is guaranteed also for small $q$ when a ball with a
sufficiently large radius $R$ is contained in $\O.$ Indeed, consider
$u\in C_0^\infty(\O)$ such that $\|u\|_\infty \le \sigma$ where
$\sigma>0$ and $\frac q {kr}|s|^{r+1}- \frac 1 {p+1}|s|^{p+1}>0$ for
any $s\in ]0,\sigma[.$ We set $u_t=u(\frac\cdot t)$, and we suppose
that $\operatorname{Supp}(u_t)=t \operatorname{Supp} (u)\subset \O$.
By a straight computation, using \eqref{eq:upperest}, we have that
    \[
        I_{q,r}(u_t)\le
        t\left(\frac{1}{2}+\frac{1}{4k^2}\right) \int_\O |\nabla u|^2dx
        - t^3 \left(\frac{q}{kr}\int_\O
        |u|^{r+1}\,dx-
    \frac{1}{p+1}\int_{\Omega} |u|^{p+1}dx\right)
    \]
where this last sum is negative for $t$ sufficiently large. As a
consequence, we should have a mountain pass solution for $2r=p+1$ and
also a minimum solution for $2r<p+1.$

\subsection{The nonexistence result}
Here we assume that $1<p<5\le r.$ Following \cite{St}, we adapt the
Poho\u{z}aev arguments in \cite{P} to our situation (for a similar
result see also \cite{DM}).
\\
Actually the proof is the same as in \cite{AD}, but we report it
here for completeness.

Let $\O \subset \RT$ be a star shaped domain and $u,\Phi\in C^2(\O)\cap C^1(\bar\O)$ be a nontrivial
solution of (\ref{Pq}). If we multiply the first equation of (\ref{Pq}) by $x\cdot\n u$ and the second one by $x\cdot\n\Phi$ we have that
\begin{align*}
0=&(\Delta u + q \Phi |u|^{r-2} u - |u|^{p-1} u )(x\cdot\n u)\\
 =&\operatorname{div}\left[(\n u)(x\cdot\n u)\right] - |\n u|^2
   - x\cdot \n \left(\frac{|\n u|^2}{2}\right)
   + \frac{q}{r} x\cdot\n\left(\Phi |u|^r\right) - \frac{q}{r} (x\cdot\n\Phi)|u|^r
   - \frac{1}{p+1} x \cdot \n (|u|^{p+1})\\
 =&\operatorname{div}\left[(\n u)(x\cdot\n u) - x \frac{|\n u|^2}{2}
   +\frac{q}{r} x \Phi |u|^r - \frac{1}{p+1} x |u|^{p+1}\right] \\
   &\qquad +\frac{1}{2} |\n u|^2
   - \frac{3}{r} q \Phi |u|^r - \frac{q}{r} (x\cdot\n\Phi) |u|^r
   + \frac{3}{p+1} |u|^{p+1}
\end{align*}
and
\begin{align*}
0=&(r\Delta \Phi + 2q |u|^r)(x\cdot\n \Phi)\\
 =&r\operatorname{div}\left[(\n \Phi)(x\cdot\n \Phi)\right] - r|\n \Phi|^2
   - \frac{r}{2} x\cdot \n \left(|\n \Phi|^2\right) + 2q (x\cdot\n\Phi) |u|^r\\
 =&r\operatorname{div}\left[(\n \Phi)(x\cdot\n \Phi) - x \frac{|\n \Phi|^2}{2}
   \right] + \frac{r}{2} |\n \Phi|^2 + 2q (x\cdot\n\Phi) |u|^r.
\end{align*}
Let ${\bf n}$ be the unit exterior normal to $\partial\O$.
Integrating on $\O$, since by boundary conditions $\nabla u = \frac{\partial u}{\partial {\bf n}} {\bf n}$ and $\nabla \Phi = \frac{\partial \Phi}{\partial {\bf n}} {\bf n}$ on $\partial \Omega$, we obtain
\begin{equation}\label{eq:Poho1}
-\frac{1}{2} \|\n u \|_2^2 - \frac{1}{2} \int_{\partial\O} \left| \frac{\partial u}{\partial {\bf n}}\right|^2 x\cdot {\bf n} = -\frac{3}{r} q \int_\O \Phi |u|^r -\frac{q}{r} \int_\O (x\cdot\n\Phi) |u|^r + \frac{3}{p+1} \|u\|_{p+1}^{p+1}
\end{equation}
and
\begin{equation}\label{eq:Poho2}
-\frac{r}{2} \|\n\Phi\|_2^2 -\frac{r}{2} \int_{\partial\O} \left| \frac{\partial \Phi}{\partial {\bf n}}\right|^2 x\cdot {\bf n}
= 2q \int_\O (x \cdot \n \Phi) |u|^r.
\end{equation}
Substituting \eqref{eq:Poho2} into \eqref{eq:Poho1} we have
\begin{equation}\label{eq:Pohoc}
-\frac{1}{2} \|\n u \|_2^2 - \frac{1}{2} \int_{\partial\O} \left| \frac{\partial u}{\partial {\bf n}}\right|^2 x\cdot {\bf n} = -\frac{3}{r} q \int_\O \Phi |u|^r + \frac{1}{4}  \|\n\Phi\|_2^2 + \frac{1}{4}  \int_{\partial\O} \left| \frac{\partial \Phi}{\partial {\bf n}}\right|^2 x\cdot {\bf n} + \frac{3}{p+1} \|u\|_{p+1}^{p+1}.
\end{equation}
Moreover, multiplying the first equation of (\ref{Pq}) by $u$ and the second one by $\Phi$ we get
\begin{equation}\label{eq:Ne1}
\|\n u \|_2^2=  q \int_\O \Phi |u|^r - \|u\|_{p+1}^{p+1}
\end{equation}
and
\begin{equation}\label{eq:Ne2}
r\|\n\Phi\|_2^2 = 2q \int_\O \Phi |u|^r.
\end{equation}

Hence, combining \eqref{eq:Pohoc}, \eqref{eq:Ne1} and \eqref{eq:Ne2}, we have
\[
\frac{r-5}{4}\|\n\Phi\|_2^2+\frac{5-p}{2(p+1)} \|u\|_{p+1}^{p+1} + \frac{1}{2} \int_{\partial\O} \left| \frac{\partial u}{\partial {\bf n}}\right|^2 x\cdot {\bf n}  + \frac{1}{4}  \int_{\partial\O} \left| \frac{\partial \Phi}{\partial {\bf n}}\right|^2 x\cdot {\bf n} = 0
\]
and we get a contradiction.

\end{document}